\documentclass[twoside,12pt]{article}
\usepackage{graphics}
\usepackage{graphicx}

\usepackage{amssymb,amsfonts}
\usepackage{amsmath}
\usepackage{amsthm}

\usepackage{amssymb,latexsym}
\usepackage{cite} 

\theoremstyle{plain}
\newtheorem{theorem}{Theorem}
\newtheorem{lemma}{Lemma}

\theoremstyle{definition}

\theoremstyle{remark}

\numberwithin{equation}{section}

\usepackage{graphics}
\usepackage{graphicx}

\def\bb{$\lim\limits_{i\to \infty }{\frac{b_{i+1}}{b_{i}}}$}
\def\aa{$\lim\limits_{i\to \infty }{\frac{a_{i+1}}{a_{i}}}$}
\def\ba{$\lim\limits_{i\to \infty }{\frac{b_{i}}{a_{i}}}$}
\def\bab{$\lim\limits_{i\to \infty }{\frac{b_{i}}{a_{i}+b_{i}}}$}
\def\abab{$\lim\limits_{i\to \infty }{\frac{a_{i+1}+b_{i+1}}{a_{i}+b_{i}}}$}

\def\balpha{\text{\boldmath $\alpha$}}

\author{L\'aszl\'o N\'emeth}
\date{}

\title{\bf Trees on hyperbolic honeycombs\thanks{AMS: 05C05, 52C20;
          key words: hyperbolic mosaics, regular mosaics, honeycombs, graph, tree;
          lnemeth@emk.nyme.hu}}

\begin{document}
\maketitle

\begin{abstract}
In the hyperbolic plane there are infinite regular lattices. From a fix vertex of a lattice tree graphs can be constructed recursively to the next layers with edges of the lattice. In this article we examine the properties of the growing of trees and the probabilities of length of trees considering the vertices on level $i$.

\end{abstract}

\section{Introduction}\label{sec:introduction}

In a regular mosaic we can define belts of cells around a fix vertex of the mosaic. Belt $0$ is the fix vertex. The first belt consists of the cells of the mosaic having common (finite) points with the fix vertex. If belt $i$ is known, let belt $(i+1)$ consist of the cells that have a common (finite) point (not necessarily a common vertex) with the belt $i$, but have not with the belt $(i-1)$. Figure \ref{abra:tree45} shows the  first three belts in mosaic $\{4,5\}$. Earlier studies have dealt with the problem of the growing of belts.  Let $v_i$ be the number of the cells in the belt $i$. The crystal-growing ratio, $\lim_{i\to \infty }{({v_{i+1}}/{v_{i}})}$, is known for all 2-dimensional and some 3- and 4-dimensional regular mosaics in  hyperbolic spaces (\cite{kar, hor, nem1, nem2, ver2, zei}). 

In this article we consider a regular planar mosaic with Schl\"af{}li's symbol $\{p,q\}$ (\cite{Cox}) as a lattice and construct tree graphs along the edges . The number of the trees grows from belt to belt and we examine the intensity of this growth. If $(p-2)(q-2)=4$, then the lattice is Euclidean, while for $(p-2)(q-2)>4$ the lattice is hyperbolic. There are only three regular lattices ($\{3,6\}$, $\{4,4\}$, $\{6,3\}$) on the Euclidean plane but there are infinite ones on the hyperbolic plane. (Some papers have studied percolation problems on hyperbolic lattices, where mosaics are considered to be lattices \cite{Baek, Hang}.)

 \section{Trees}\label{sec:trees}

Let us fix a $B_0$ vertex of the lattice as a main root (label it layer $0$ or level $0$). Let the outer boundary of belt $i$ be layer $i$ or level $i$. Now we connect the vertices from layer $0$ to layer 1 along the edges of the lattice. We build the trees from level $(i-1)$ to level $i$ using the maximum number of edges between level $(i-1)$ and  level $i$. (All vertices on level $i$ are connected to only one vertex of the previous level. We do not let leaves on level $(i-1)$.) We never connect edges on the same layer. The rest vertices on layer $i$ will be  also roots of new trees. In this recursive way, we obtain infinitely long trees. Let $B$ denote the roots and $A$ the other vertices. In Figure~\ref{abra:tree45} and \ref{abra:tree37} the thick edges show the trees from level 0 to level 4. (The dual problem is the case when we get trees by connecting the centres of the cells of the mosaic (Figure~\ref{abra:tree45dual}).) 

 \begin{figure}[!htb]
 \centering   \includegraphics{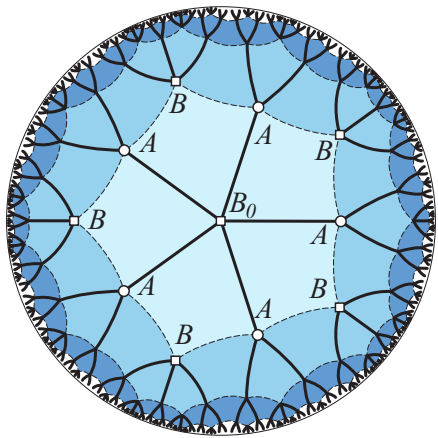} \caption{\emph{Trees of the mosaic $\{p,q\}=\{4,5\}$.}}\label{abra:tree45}
 \end{figure}
 
 \begin{figure}[!htb]
 \centering   \includegraphics{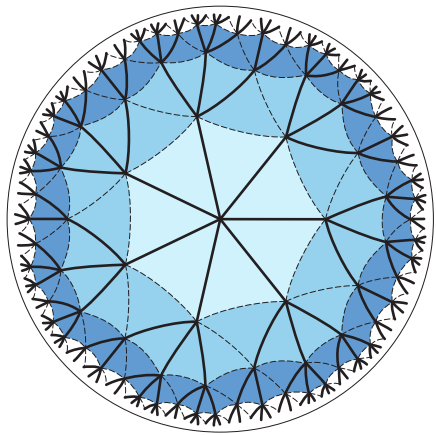} \caption{\emph{Trees of the mosaic $\{p,q\}=\{3,7\}$.}}\label{abra:tree37}
 \end{figure}

In case of $q=3$ there is not any tree with this definition, because only one edge is not enough to connect the layers. If $p=3$ the algorithm does not give roots except the main one (Figure~\ref{abra:tree37}). Let $a_i$ and $b_i$ be the numbers of the vertices $A$ and $B$ on level $i$, respectively.

 \begin{figure}[!h]
 \centering   \includegraphics{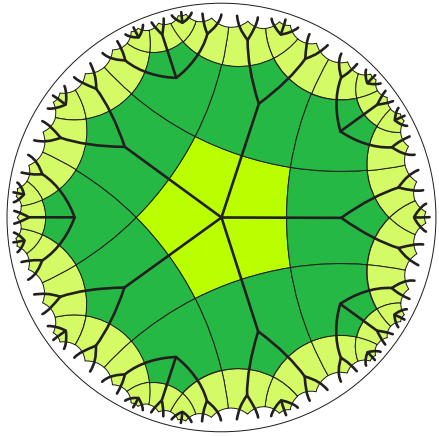} \caption{\emph{Trees of the mosaic $\{5,4\}$, dual of mosaic  $\{4,5\}$.}} \label{abra:tree45dual}
 \end{figure}

In the following we shall give some properties of the sequences $a_i$ and $b_i$ for all hyperbolic planar lattices $\{p,q\}$ except the case $p=3$ or $q=3$.

In case of all $\{p,q\}$ lattices $a_0=0$, $b_0=1$ and $a_1=q$, $b_1=q(p-3)$.

\begin{lemma}\label{lemma:ai} If $p>3$, $q>3$ and  $i\geq 1$, then $a_{i+1}=(q-3)a_{i}+(q-2)b_{i}$ and \\ $b_{i+1}=((q-3)(p-3)-1)a_{i}+((q-2)(p-3)-1)b_{i}$.
\end{lemma}

\begin{proof}
The degrees of all vertices are $q$. On level $i$ a vertex has an edge from level $(i-1)$ and two on level $i$. So, this vertex is connected to the next level with $q-3$ edges to $q-3$ different $A$ (in Figure \ref{abra:tree_i} we can see a part of level $i$ and $(i+1)$, where a circle denotes a vertex $A$, a square a vertex $B$ and a triangle a vertex whose type is unknown).  Similarly, a root has two edges on level $i$, so it is connected to $q-2$ new vertices $A$ on the next level.

\begin{figure}[!htb]
\centering   \includegraphics{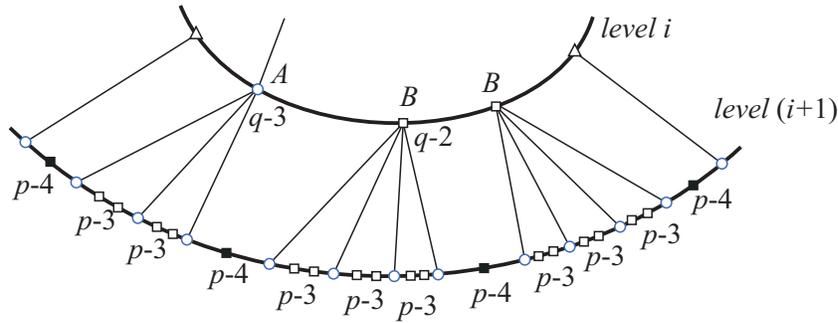} \caption{\emph{Vertices and edges from level $i$ to level $(i+1)$.}}\label{abra:tree_i}
\end{figure}

All vertices $A$ are rounded by $q$ pieces of $p$-gons. Two of them are in the belt $i$, the others are in the belt $(i+1)$. Among them there are $q-4$ pieces of $p$-gons that have $p-3$ vertices which are not connected to the tree of $A$ but they are on level $(i+1)$. So they are roots in the next level. Two further $p$-gons can have vertices as roots, but they are connected not only to $A$ but also to other vertices on level $i$ (filled squares in Figure \ref{abra:tree_i}). Due to the multiplicity  we calculate only half of them to the vertex $A$. Then the number of roots on level $(i+1)$ calculated from a vertex $A$ is $(q-4)(p-3)+2(p-4)/2=$ $(q-3)(p-3)-1$.

We can determine the number of new roots in a similar way in case of a vertex $B$. Now there are $q-3$ pieces of $p$-gons and all have $p-3$ vertices which are new roots. Similarly to case $A$, there are two polygons around $B$ having an other vertex on level $i$ (filled squares in Figure \ref{abra:tree_i}). Due to the multiplicity we have to divide their number by two to obtain the correct number of new roots. So the new roots are $(q-3)(p-3)+2(p-4)/2=$ $(q-2)(p-3)-1$ altogether. \end{proof}

From Lemma \ref{lemma:ai} we derive a recursive equation system $(p>3,\, q>3,\, i\geq1)$
\begin{eqnarray}
a_{i+1}&=&(q-3)a_{i}+(q-2)b_{i}\label{eq:a}\\
b_{i+1}&=&\big((q-3)(p-3)-1\big)a_{i}+\big((q-2)(p-3)-1\big)b_{i},
\end{eqnarray}
\noindent that can also be written in a matrix form
\begin{eqnarray}\mathbf{w}_{i+1}=\mathbf{M}\mathbf{w}_i,\end{eqnarray}
where $\mathbf{w}_i=[a_i\ \ b_i]^T$, 
$\mathbf{M}=
\bigg( \begin{matrix} 
 q-3& q-2\\
 (q-3)(p-3)-1 &(q-2)(p-3)-1\\
\end{matrix}\bigg)$.

\noindent All $\left\{ r_i\right\}_{i=1}^\infty$ recursive sequences are defined by
\begin{eqnarray}
r_{i}&=&{\balpha} ^T\mathbf{w}_i, \label{eq:ra}
\end{eqnarray}
\noindent and can be determined explicitly as (\cite{nem1, nem2})
\begin{eqnarray}
r_i=g_{r1}z_1^i+g_{r2}z_2^i, \label{eq:gyokokkel}
\end{eqnarray}
\noindent where $\mathbf{\balpha}$ is a real vector and if $c=(p-2)(q-2)-2>2$ then $z_1=\frac{c+\sqrt{c^2-4}}{2}$, $z_2=\frac{c-\sqrt{c^2-4}}{2}$. The quantities $z_1$, $z_2$ are the eigenvalues of matrix $\mathbf{M}$ and  $z_1=\left|z_1\right|>\left|z_2\right|\neq 0 $ with  
\begin{eqnarray*}
g_{r1}=\frac{r_2-z_2r_1}{z_1(z_1-z_2)}\neq 0,\ \   g_{r2}=\frac{z_1r_1-r_2}{z_2(z_1-z_2)}. \label{eq:gi}
\end{eqnarray*}

\begin{theorem}
The growing ratios of the vertices, of the roots and in addition,  of all the vertices are equal to $z_1$, \aa $=$ \bb $=$ \abab$=z_1$ $(i\geq 1)$.
\end{theorem}
\begin{proof}
Let sequences $r_i$ be equal to $a_i$, $b_i$ or $a_i+b_i$. Then $\balpha^T=[0 \ 1]^T\mathbf{M}^{-1}$, $\balpha^T=[1 \ 0]^T\mathbf{M}^{-1}$ or $\balpha^T=[1 \ 1]^T\mathbf{M}^{-1}$, respectively, and from \cite{nem1} we obtain that the limits are equal to the largest eigenvalue $z_1$ of matrix  $\mathbf{M}$.   \end{proof}

\medskip

With the help of equation \eqref{eq:gyokokkel}, we can write $a_i=g_{a1}z_1^i+g_{a2}z_2^i$, $b_i=g_{b1}z_1^i+g_{b2}z_2^i$, $a_i+b_i=g_{ab1}z_1^i+g_{ab2}z_2^i$. Let $L=\frac{g_{b1}}{g_{a1}}$ and $K=\frac{L}{1+L}$.

\begin{theorem}
$\lim\limits_{i\to \infty }{\frac{b_{i}}{\sum_{j=0}^{i}b_{j}}}=\frac{z_1-1}{z_1}$  $(i\geq 1)$. \ba$=L$,  \bab$=K$  $(i\geq 1)$.
\end{theorem}
\begin{proof}

The first limit for $r_i=b_i$ comes from \cite{nem1}.

As $\lim\limits_{i \to \infty } \left({\frac{z_2}{z_1}}\right)^i =0$, then
\begin{eqnarray}
 \lim\limits_{i \to \infty }{\frac{b_{i}}{a_i}}=
 \lim\limits_{i \to\infty}{\frac{g_{b1}z_1^{i}+g_{b2}z_2^{i}}
         {g_{a1}z_1^i+g_{a2}z_2^i}}
 =\lim\limits_{ i\to\infty}
 {\frac{g_{b1}+g_{b2}\left(\frac{z_2}{z_1}\right)^i}
    {g_{a1}+g_{a2}\left(\frac{z_2}{z_1}\right)^i}}= \frac{g_{b1}}{g_{a1}}=L.
\end{eqnarray}

\begin{eqnarray}
 \lim\limits_{i \to \infty }{\frac{b_{i}}{a_i+b_i}}=   
   \lim\limits_{i \to \infty }{\frac{\frac{b_i}{a_i}}{1+\frac{b_i}{a_i}}}= K.\ 
\end{eqnarray}
\end{proof}

\section{Probability}

In what follows let us suppose that $i$ is large enough. We consider a vertex $V$ from level $i$. Let $p_{i,j}$ $(0\leq j\leq i)$ be  the probability that the root of the vertex $V$ on level $i$ is on level $j$ and let $p_{i,-j}$  be the probability that the root of the vertex $V$ on level $i$ is on level $j$ or on a level below. Let  $M=\frac{h L}{1+h L}$, where $h=\frac{q-2}{q-3}$ and  $q>3$.

\begin{theorem} \label{theorem:probability}
If  $0<j<i$, then $p_{i,-j}=(1-K)(1-M)^{i-j-1}$ and $p_{i,j}=(1-K)M(1-M)^{i-j-1}$. Moreover $p_{i,i}=K$ and $p_{i,0}=p_{i,-0}=(1-K)(1-M)^{i-1}$.
\end{theorem}

\begin{proof}
The ratio of the roots and all vertices on level $i$ is $\frac{b_{i}}{a_i+b_i}$. If $i$ is large enough, then $\frac{b_{i}}{a_i+b_i} \approx K$ (in case of mosaic $\{4,5\}$ if $i=7$ then the difference between the two values is less than $10^{-6}$). Similarly, $\frac{b_{i-1}}{a_{i-1}+b_{i-1}} \approx K$. Let $V$ be a vertex on level $i$. Thus the probability that the vertex $V$ is a root is $K$ and that $V$ is not a root is $1-K$. So $p_{i,i}=K$ and $p_{i,-(i-1)}=1-K$.

On level $(i-1)$ the ratio of the numbers of roots and other vertices is $\frac{b_{i-1}}{a_{i-1}}\approx L$. Equation \eqref{eq:a} implies that $\frac{q-2}{q-3}\, L$ gives the ratio between the numbers of vertices with roots on level $(i-1)$ and with roots below. If $k_1$ is the number of the vertices which have roots on level $(i-1)$, then 
$\frac{k_1}{a_1}\approx M$. Thus the probability that the root of $V$ is neither on level $i$, nor on level $(i-1)$ is $(1-K)(1-M)$, therefore $p_{i,-(i-2)}=(1-K)(1-M)$.

Similarly, if $k_2$ is the number of the vertices on level $i$ whose roots are on level $(i-2)$ then $\frac{k_2}{a_1-k_1}\approx M$. Thus $p_{i,-(i-3)}=(1-K)(1-M)^{2}$. 

Generally, we obtain  $p_{i,-j}=(1-K)(1-M)^{i-(j+1)}$, here $0\leq j<i$. And in case of $0<j<i$, $p_{i,j}=(1-K)(1-M)^{i-(j+1)}-(1-K)(1-M)^{i-j}=(1-K)(1-M)^{i-j-1}\big(1-(1-M)\big)=(1-K)M(1-M)^{i-j-1}$.
\end{proof}
\bigskip

The probabilities in Theorem \ref{theorem:probability} are more precise the higher the $i$ is and the closer the $j$ is to $i$.

\begin{figure}[!htb]
 \centering   \includegraphics{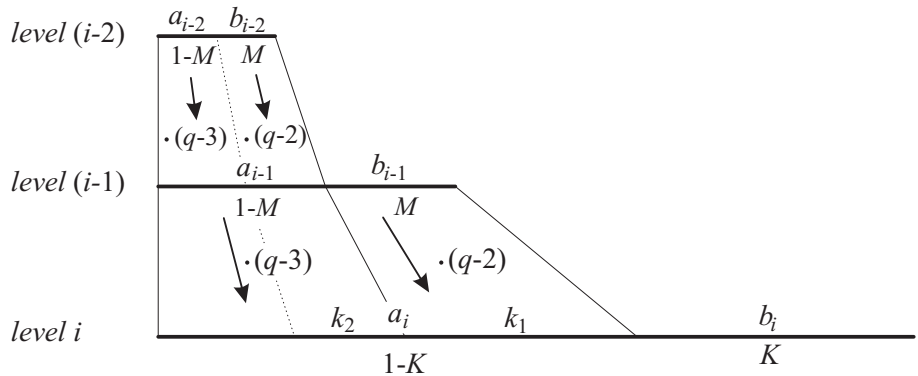} \caption{\emph{Probability of vertices on level $i$.}}\label{abra:probability}
\end{figure}

\noindent \textbf{Examples for lattice $\{4,5\}$.}

For lattice $\{4,5\}$ (Figure~\ref{abra:tree45}) we give all the results discussed above. The numbers of vertices and roots on level $i$ are  in Table \ref{tablazat:number_of_vertices} $(0\leq i\leq 10)$.

\begin{table}[!ht]
  \centering \setlength{\tabcolsep}{0.4em}
\begin{tabular}{|c|c|c|c|c|c|c|c|c|c|c|c|}
  \hline
 $i$   &  0  &  1 & 2  & 3  & 4   & 5   & 6    &  7    & 8     & 9      & 10   \\ \hline \hline
 $a_i$ &  0  &  5 & 25 & 95 & 335 & 1325& 4945 & 18455 & 68875 & 257045 & 959305  \\ \hline
 $b_i$ & 1   &  5 & 15 & 55 & 205 & 765 & 2855 & 10655 & 39765 & 148405 & 553855 \\ \hline
 $a_i+b_i$& 1 & 10 & 40 & 150 & 560 & 2090 & 7800 & 29110 & 108640 & 405450  & 1513160   \\ \hline
 \end{tabular}
\caption{\emph{Numbers of vertices and roots on level $i$}\label{tablazat:number_of_vertices}}
\end{table}

The recursion matrix of growing is  
 $\mathbf{M}=
  \begin{pmatrix}
       2 & 3 \\
     1   & 2\\
  \end{pmatrix}$ 
and the crystal-growing ratio is $z_1=2+\sqrt{3}\approx 3.732051$  (while the other eigenvalue is $z_2=2-\sqrt{3}$). In case of $i=10$ the difference between the corresponding ratios and their limits are less then $10^{-9}$ (in case of $i=100$ the difference is less then  $10^{-113}$).\\ 
The values of $g_{r1}$ for the recursive sequences $a_i$, $b_i$, $a_i+b_i$ are 
$g_{a1}= -\frac{5}{2}+\frac{5}{2}\sqrt{3} \approx 1.830127$, 
$g_{b1}=  \frac{5}{2}-\frac{5}{6}\sqrt{3} \approx 1.056624 $,
$g_{ab1}= \frac{5}{3}\sqrt{3} \approx 2.886751 $.
Other important values are $L = \frac13{\sqrt{3}}\approx 0.577350 $, 
$K = \frac{1}{2}(\sqrt{3}-1)\approx 0.366025$  and $M=-3+2\sqrt{3}\approx 0.464102$.  The difference between the value of $K$ ($p_{i,i}\approx K$) and the exact probability $p_{10,10}=\frac{b_{10}}{a_{10}+b_{10}}$ is less then $10^{-10}$.

For the probabilities from  Theorem \ref{theorem:probability} we obtain the values in Table \ref{tablazat:probability}.

\begin{table}[!ht]
  \centering
\begin{tabular}{|c|c|c|}
  \hline
   $j$   &  $i=7$  & $i=10$   \\
  \hline \hline
 10  &  -  &  0.366025   \\ \hline
 9  &  -  &  0.294228   \\ \hline
 8  &  -  &  0.157677   \\ \hline 
 7  &  0.366025  &  0.084499   \\ \hline
 6  &  0.294229  &  0.045283   \\ \hline
 5  &  0.157677  &  0.024267   \\ \hline  
 4  &  0.084499  &  0.013005   \\ \hline
 3  &  0.045283  &  0.006969   \\ \hline
 2  &  0.024267  &  0.003735   \\ \hline
 1  &  0.013005  &  0.002001   \\ \hline
 0  &  0.015016  &  0.002311   \\ \hline
\end{tabular}
\caption{\emph{Probabilities $p_{i,j}$}\label{tablazat:probability}}
\end{table}

We can give the number of the vertices on level $i$ which have the common main root given by the expression $s_i=q(q-3)^{(i-1)}$. Then the exact probability is $p_{i,0}=\frac{s_i}{a_i+b_i}$. If $i=7$ or $i=10$, then $p_{7,0}=\frac{320}{29110}\approx 0.010993$ or $p_{10,0}=\frac{2560}{1513160}\approx 0.001692$ and we obtain that $O(p_{7,0})\approx 10^{-3}$ or $O(p_{10,0})\approx 10^{-4}$ (in case of $i=100$,  $O(p_{100,0})\approx 10^{-28}$).

We can conclude that the probabilities from Theorem \ref{theorem:probability} are exact enough even in case of $i=7$. The worst result is in case of $j=0$, but as $j$ is getting closer to $i$ the result is getting more precise.

\bigskip 
\noindent \textbf{Remarks}

 We can join the main trees in their common main root creating an infinite main tree without root and we can connect the other roots to the main tree by edges, this way we obtain a spanning tree of the vertices of the lattice. In case of $p=3$ the definition in the introduction always gives a spanning tree and we can go back to the main root from all vertices on level $i$.

There is only one regular Euclidean planar lattice when $p>3$ and $q>3$. Now we examine that $\{4,4\}$ lattice (Figure~\ref{abra:tree44}). The $a_0=0$, $b_0=1$, $a_i=8i-4$ and $b_i=4$ is constant. As $z_1=z_2=1$, the growing of the vertices from level to level is slow, if $i$ is large enough, the growth is almost constant.  The probability that the root of a vertex $V$ on level $i$  is on level $j$ $(0<j<i)$ is $p_{i,-j}=\frac{8j+4}{8i}=\frac{j+\frac12}{i}$ and 
$\lim\limits_{i\rightarrow \infty}p_{i,-j}=0$.

\begin{figure}[!htb]
\centering   \includegraphics{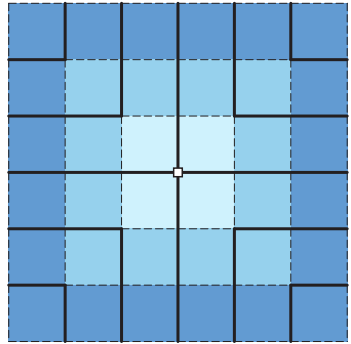} \caption{\emph{Trees of the mosaic $\{p,q\}=\{4,4\}$.}}\label{abra:tree44}
\end{figure}


\end{document}